\documentclass[final,1p,times,number]{elsarticle}
 \journal{Advances in Applied Mathematics}          
\usepackage[colorlinks]{hyperref}

\usepackage{amsmath,amscd,amssymb,amsthm}
\usepackage[all,cmtip,matrix,arrow,curve]{xy}
\usepackage{tikz}

\theoremstyle{plain}
\newtheorem{theorem}{Theorem}[section]

\theoremstyle{definition}
\newtheorem{definition}[theorem]{Definition}

\newtheorem{example}[theorem]{Example}

\theoremstyle{remark}

\RequirePackage{color}

\definecolor{todo}{rgb}{1,0,0}
\definecolor{answer}{rgb}{0,0,1}
\definecolor{new}{rgb}{1,0,1}
\definecolor{conditional}{rgb}{0,1,0}
\definecolor{e-mail}{rgb}{0,.40,.80}
\definecolor{reference}{rgb}{.20,.60,.22}
\definecolor{mrnumber}{rgb}{.80,.40,0}
\definecolor{citation}{rgb}{0,.40,.80}

\newcommand{\Le}{\leqslant}
\newcommand{\Ge}{\geqslant}

\newcommand{\K}{K}

\newcommand{\GL}{\mathbf{GL}}

\newcommand{\Mn}{\mathbf{M}}

\newcommand{\Quot}{\mathrm{Quot}}
\def\Q{{\mathbb Q}}
\def\Z{{\mathbb Z}}

\begin{document}

\begin{frontmatter}


\author{Alexey Ovchinnikov}
\address{CUNY Queens College, Department of Mathematics, 65-30 Kissena Blvd,
Queens, NY 11367-1597, USA\\
CUNY Graduate Center, Department of Mathematics, 365 Fifth Avenue,
New York, NY 10016, USA
}
\ead{aovchinnikov@qc.cuny.edu}


\title{Difference integrability conditions for parameterized linear difference and differential equations}

\begin{abstract}
This paper is devoted to  integrability conditions for systems of linear difference and differential  equations with difference parameters. It is shown that  such a system is difference isomonodromic if and only if it is difference isomonodromic with respect to each parameter separately. Due to this result, it is no longer necessary to solve non-linear difference  equations to verify isomonodromicity, which will improve efficiency of computation with these systems.\end{abstract}

\end{frontmatter}

\section{Introduction}
In this paper, we improve the algorithm that verifies if a system of linear difference and differential equations with difference parameters is isomonodromic (Definition~\ref{def:isomonodromic}). Given a system of difference or differential equations with parameters, it is natural to ask if its solutions satisfy extra linear difference equations with respect to the parameters. 
Consider  an algorithm whose input consists of a field $\K$ with commuting automorphisms $\phi_1,\ldots,\phi_q,\sigma_1,\ldots,\sigma_r$ and derivations $\partial_1,\ldots,\partial_m$ on $\K$ and a system of difference equations 
\begin{equation}\label{eq:difference1}
\phi_1(Y)=A_1Y,\ldots,\phi_q(Y)=A_qY,\quad \partial_1Y=B_1Y,\ldots,\partial_mY=B_mY,
\end{equation} where $A_i\in \GL_n(\K)$ and $B_j \in \Mn_n(\K)$, $1\Le i\Le q$, $1\Le j\Le m$. Its output consists of such additional linear difference equations
\begin{equation}\label{eq:differenceextra}
\sigma_1Y=C_1Y,\ldots,\sigma_rY=C_rY,
\end{equation} where $C_i\in \GL_n(\K)$, $1\Le i\Le r$, if they exist, for which there exists an invertible matrix solution of~\eqref{eq:difference1} that satisfies~\eqref{eq:differenceextra} as well. 
For an invertible matrix solution of~\eqref{eq:difference1} to possibly exist, the $A_i$'s and $B_j$'s must satisfy the integrability conditions~\eqref{eq:intdifference},~\eqref{eq:intdifferential1}, and~\eqref{eq:intdifferential2}. Moreover, the existence of the $C_i$'s above is equivalent to the existence of $C_i$'s satisfying another large collection of integrability conditions~\eqref{eq:intdiffdiff1},~\eqref{eq:intdiffdiff2} and~\eqref{eq:intdiffdiff3}. Such systems~\eqref{eq:difference1} are called isomonodromic in analogy with differential equations \cite{Sabbah,Sibuya,JM,JMU,GO}. Isomonodromy problems for $q$-difference equations and their relations with $q$-difference Painlev\'e equations were studied in~\cite{Joshi1,Joshi2}.

The main result, Theorem~\ref{thm:main}, states that~\eqref{eq:intdiffdiff3}, which are non-linear difference equations in the $C_i$'s, do not have to be verified to check the existence of the $C_i$'s. More precisely,  the existence of $C_i$'s satisfying~\eqref{eq:intdiffdiff1} and~\eqref{eq:intdiffdiff2} implies the existence of new  invertible matrices that satisfy all of~\eqref{eq:intdiffdiff1},~\eqref{eq:intdiffdiff2}, and~\eqref{eq:intdiffdiff3}.

Since, due to our result, we only need to check the existence of solutions (that have entries in the ground field $\K$) of a system of linear difference and differential equations, a complexity estimate 
for verifying whether a system of difference and differential equations is isomonodromic becomes possible due to~\cite{AGKCASC,AGK,BCEW,BCBC,BGS,BGS2007}.

A similar problem but for systems of differential equations was considered in~\cite{GO}, motivated by the classical results~\cite{JM,JMU}. Differential categories developed in~\cite{GGO} formed the main technical tool in~\cite{GO}. On the contrary, our proofs are constructive, which makes them more suitable for practical use. In particular, from our proof, one can derive an algorithm that, given a common invertible solution of~\eqref{eq:intdiffdiff1} and~\eqref{eq:intdiffdiff2}, computes a common invertible solution of all~\eqref{eq:intdiffdiff1},~\eqref{eq:intdiffdiff2}, and~\eqref{eq:intdiffdiff3}.

If one attempts to find a common solution for all of~\eqref{eq:intdiffdiff1},~\eqref{eq:intdiffdiff2}, and~\eqref{eq:intdiffdiff3} in a naive way,  one might have to adjoin new elements to the base field $\K$ that are not constant with respect to the derivations and automorphisms, which is not desirable. On the other hand, Theorem~\ref{thm:main} shows how one can limit, in a computable way, these newly adjoint elements to those that are constant with respect to all derivations and the automorphisms that do not correspond to the parameters, which is another advantage of this approach. 
This is also different from the approach taken in \cite{BOS} to give an explicit treatment of a similar problem but for differential parameters, where a linearly differentially closed assumption is imposed on the field of constants with respect to the automorphisms.

The main result is expected to have further applications. In parameterized differential Galois theory~\cite{PhyllisMichael}, there are several algorithms for computing Galois groups. For $2\times 2$ systems, they are given in~\cite{Carlos,Dreyfus}. An algorithm, more general in terms of the order of the system,~\cite{MiOvSi,MiOvSi2} uses the differential analogues~\cite{GO} of our results to make the computation more efficient. 
Our results may prove useful in the design of an algorithm for computing parameterized difference Galois groups with difference parameters, as it has happened in the differential case, once \cite{DHW} or \cite{OW} is generalized to the case of several difference parameters, which is a challenge on its own.

These and other difference and differential Galois theories \cite{CharlotteMichael,CharlotteComp,CharlotteLucia12,DiVizioHardouin:DescentandConfluence,HDV,Takano} have been developed to compute all difference and differential algebraic relations that solutions of systems of linear difference and differential equations satisfy and have found many practical applications.

The paper is organized as follows. Notation is introduced and the basic notions of differential and difference algebra are reviewed in Section~\ref{sec:basicdefinitions}. The integrability conditions are discussed in Section~\ref{sec:intconddef}, where the main result is illustrated by showing a concrete example, in which we perform computation with the basic hypergeometric $q$-difference equation with parameters (Example~\ref{ex:HG}). The main result and its proof are described in Section~\ref{sec:mainresult}.
\section{Basic definitions}\label{sec:basicdefinitions}
A $\Delta$-ring is a commutative
associative ring with unit $1$ together with a set $\Delta=\{\partial_1,\dots,\partial_m\}$
of commuting derivations $\partial_i\colon R\to R$ such that
$$
\partial_i(a+b) = \partial_i(a)+\partial_i(b),\quad \partial_i(ab) =
\partial_i(a)b + a\partial_i(b),\quad a,b\in R,\ \ 1\Le i\Le m.
$$
For example, $\Q$ is a $\{\partial\}$-field 
with the unique
 zero derivation. For every $f \in \mathbb{C}(x)$, there exists a unique derivation
$\partial : \mathbb{C}(x)\to \mathbb{C}(x)$ with $\partial(x)=f$, turning $\mathbb{C}(x)$ into a $\{\partial\}$-field.

For a ring $R$, $\Mn_n(R)$ denotes the set of $n\times n$ matrices with entries in $R$, and $\GL_n(R)$ are the invertible matrices in $\Mn_n(R)$.

A $\Phi$-ring $R$ is a commutative associative ring with unit $1$ and a set $\Phi = \{\phi_1,\ldots,\phi_q\}$ of commuting automorphisms $\phi_i: R\to R$, $1\Le i\Le q$. A $\{\Phi,\Delta\}$-ring $R$ is a $\Phi$-ring and a $\Delta$-ring such that, for all $\phi\in\Phi$ and $\partial\in\Delta$, $\phi\partial=\partial\phi$. 
For subsets $\Phi'\subset \Phi$ and $\Delta'\subset \Delta$, let 
$$
R^{\Phi',\Delta'}:=\big\{r\in R\:|\: \phi(r)=r,\ \delta(r) =0,\ \text{for all}\ \phi\in\Phi',\ \delta\in\Delta'\big\},
$$
the set of $\{\Phi',\Delta'\}$-constants of $R$.

To discuss difference parameters in systems of difference and differential equations, it will be convenient for us to split the set of automorphisms into two subsets. In such a case, we will write $\{\Phi\cup\Sigma,\Delta\}$ to emphasize that the set $\Phi\cup\Sigma$ of given automorphisms is split into two subsets, $\Phi$ and $\Sigma$. In this case, the $\Sigma$ corresponds to the parameters.

Let $\Sigma=\{\sigma_1,\ldots,\sigma_r\}$, $R$ be a $\Sigma$-ring, and $y_1,\ldots,y_n$ be indeterminates over $R$. The ring $${R\{y_1,\ldots,y_n\}}_\Sigma := R\left[\theta y_i\:\big|\: \theta=\sigma_1^{i_1}\cdot\ldots\cdot\sigma_r^{i_r},\ 1\Le i\Le n,\ i_j\in\Z,\ 1\Le j\Le r\right]$$ is called the ring of $\Sigma$-polynomials, which is naturally a $\Sigma$-ring.

\begin{example} To clarify the above, we will  show a few basic examples of $\{\Phi\cup\Sigma,\Delta\}$-rings:
\begin{enumerate}
\item $R = \Q(x_1,\ldots,x_m, a_1,\ldots,a_r)$ with $\Phi=\{\phi_1,\ldots,\phi_m\}$, $\Sigma=\{\sigma_1,\ldots,\sigma_r\}$, and $\Delta=\{\partial_1,\ldots,\partial_m\}$ defined by \begin{align*}
&\phi_i (f)(x_1,\ldots,x_m,a_1,\ldots,a_r) = f(x_1,\ldots,x_i+1,\ldots,x_m,a_1,\ldots,a_r),& f\in R,\ \ 1\Le i\Le m,\\
&\sigma_i (f)(x_1,\ldots,x_m,a_1,\ldots,a_r) = f(x_1,\ldots,x_m,a_1,\ldots,a_i+1,\ldots,a_r),&f\in R,\ \ 1\Le i\Le r,\\
&\partial_i = \partial/\partial x_i,&1\Le i\Le m,
\end{align*}
where $x_1,\ldots,x_m,a_1,\ldots,a_r$ are transcendental over $\Q$.
\item $R = \Q(x,a)$ with $\Phi = \{\phi\}$, $\Sigma=\{\sigma\}$, and $\Delta=\{\partial\}$ defined by
$$
\phi(f)(x,a) = f(q_1x,a),\quad\sigma(f)(x,a)=f(x,q_2a),\quad \partial = x\partial/\partial_x,
$$
where $0\ne q_1,q_2 \in\Q$ and $x$ and $a$ are transcendental over $\Q$.
\end{enumerate}
\end{example}
\begin{example} Difference and differential equations with difference parameters that appear in practice include, among many others,
$$
y''(x)+\big(x^3+ax+b\big)\cdot y(x)=0
$$
and
\begin{equation}\label{eq:HG}
y\big(q^2x\big) -\frac{(a+b)x-(1+c/q)}{abx-c/q}y(qx) + \frac{x-1}{abz -c/q}y(x) = 0.
\end{equation}
\end{example}

\section{Integrability conditions}\label{sec:intconddef}
In this section, the integrability conditions are introduced and an example illustrating our approach is given. The main result will be shown in Section~\ref{sec:mainresult}. 
Let $\K$ be a $\{\Phi\cup\Sigma,\Delta\}$-field of characteristic zero, where $\Phi=\{\phi_1,\ldots,\phi_q\}$, $\Sigma=\{\sigma_1,\ldots,\sigma_r\}$, and $\Delta= \{\partial_1,\ldots,\partial_m\}$, and $$A_1,\ldots,A_q \in \GL_n(\K)\quad\text{and}\quad B_1,\ldots, B_m \in \Mn_n(\K).$$ Consider the system of difference-differential equations 
\begin{equation}\label{eq:difference}\phi_1(Y)=A_1Y,\ldots,\phi_q(Y)=A_qY,\quad \partial_1(Y)=B_1Y,\ldots,\partial_m(Y)=B_mY. 
\end{equation}
If $L\supset \K$ is a $\{\Phi,\Delta\}$-ring extension and $Z \in \GL_n(L)$ satisfies~\eqref{eq:difference} then, for all $i$, $j$, $1\Le i,j\Le q$, we have
$$
\phi_j(A_i)A_jZ=\phi_j(\phi_i(Z))=\phi_i(\phi_j(Z))= \phi_i(A_j)A_iZ.
$$ 
Therefore, we obtain
\begin{equation}\label{eq:intdifference}
\phi_j(A_i)A_j = \phi_i(A_j)A_i
\end{equation}
Moreover,
\begin{align*}
\phi_j(B_i)A_jZ&=\phi_j(\partial_i(Z))=\partial_i(\phi_j(Z))=\partial_i(A_j)Z+A_jB_iZ
\end{align*}
and
\begin{align*}
\partial_j(B_i)Z+B_iB_jZ&=\partial_j(\partial_i(Z))=\partial_i(\partial_j(Z))=\partial_i(B_j)Z+B_jB_iZ,
\end{align*}
which imply that
\begin{equation}\label{eq:intdifferential1}
\phi_j(B_i)A_j=\partial_i(A_j)+A_jB_i,\quad 1\Le i\Le m,\ 1\Le j\Le q,
\end{equation}
and
\begin{equation}\label{eq:intdifferential2}
\partial_j(B_i)-\partial_i(B_j)=[B_j,B_i],\quad 1\Le i,j\Le m.
\end{equation}
If, in addition, there exist $C_1,\ldots, C_r \in \GL_n(\K)$ such that
\begin{equation}\label{eq:C}
\sigma_1(Z)=C_1Z,\ldots,\sigma_r(Z)=C_rZ,
\end{equation} 
then, similarly to the above,
\begin{align}\label{eq:intdiffdiff1}
\phi_j(C_i)A_j&=\sigma_i(A_j)C_i,&&1\Le i\Le r,\ 1\Le j\Le q,\\
\label{eq:intdiffdiff2}
\sigma_j(B_i)C_j&=\partial_i(C_j)+C_jB_i,&&1\Le i\Le m,\ 1\Le j\Le r,\\
\label{eq:intdiffdiff3}
\sigma_j(C_i)C_j&=\sigma_i(C_j)C_i,&&1\Le i,j\Le r.
\end{align} 
Therefore,~\eqref{eq:intdifference},~\eqref{eq:intdifferential1},~\eqref{eq:intdifferential2},~\eqref{eq:intdiffdiff1},~\eqref{eq:intdiffdiff2}, and~\eqref{eq:intdiffdiff3} are necessary conditions for the existence of a common invertible matrix solution of~\eqref{eq:difference} and~\eqref{eq:C} with entries in a field $L$. They are also sufficient. Indeed, let $$L = \K(x_{11},\ldots,x_{nn}),$$
with
\begin{align*}
\partial_i((x_{rs})) = B_i(x_{rs})&,\quad 1\Le i\Le m, \quad
\phi_j((x_{rs})) = A_j(x_{rs}),\quad 1\Le j\Le q,\\
& \sigma_k((x_{rs}))=C_k(x_{rs}),\quad 1\Le k\Le r.
\end{align*}
Then~\eqref{eq:intdifference},~\eqref{eq:intdifferential1},~\eqref{eq:intdifferential2},~\eqref{eq:intdiffdiff1},~\eqref{eq:intdiffdiff2}, and~\eqref{eq:intdiffdiff3} imply that $L$ is a $\{\Phi\cup\Sigma,\Delta\}$-field.

\begin{definition}\label{def:isomonodromic}The system of linear difference-differential equations $$\phi_1(Y)=A_1Y,\ldots,\phi_q(Y)=A_qY,\quad \partial_1(Y)=B_1Y,\ldots,\partial_m(Y)=B_mY$$ with $A_1,\ldots, A_q \in \GL_n(\K)$ and $B_1,\ldots,B_m \in \Mn_n(\K)$ satisfying~\eqref{eq:intdifference},~\eqref{eq:intdifferential1} and~\eqref{eq:intdifferential2} is called {\it difference isomonodromic} if there exist $C_1,\ldots,C_r \in \GL_n(\K)$ satisfying~\eqref{eq:intdiffdiff1}, \eqref{eq:intdiffdiff2}, and~\eqref{eq:intdiffdiff3}.
\end{definition}

Connections to analytic interpretations of a similar notion, differential isomonodromy, and parameterized Galois theory can be found, for example, in  \cite[\S5]{PhyllisMichael} and  \cite[\S6.2]{GO}.

\begin{example}\label{ex:HG} 
Consider the field $\K := \Q(x,a,b,c)$ as a $\Phi\cup\Sigma$-field, where $\Phi=\big\{\phi_q\big\}$, $\Sigma=\big\{\sigma_a,\sigma_b,\sigma_c\big\}$ and, for all $f\in\K$,
\begin{align*}
&\phi_q(f)(x,a,b,c) = f(qx,a,b,c),\quad\sigma_a(f)(x,a,b,c)=f(x,qa,b,c),\\
&\sigma_b(f)(x,a,b,c)=f(x,a,qb,c),\quad \sigma_c(f)(x,a,b,c)=f(x,a,b,qc).
\end{align*}
For Eq.~\eqref{eq:HG}, the matrices
$$
C_1 := \begin{pmatrix}
1&-a\\
\dfrac{aq(x-1)}{abqx-c}&\dfrac{aq(1-ax)+c(a-1)}{abqx-c}
\end{pmatrix},\quad
C_2 := \begin{pmatrix}
1&-b\\
\dfrac{bq(x-1)}{abqx-c}&\dfrac{bq(1-bx)+c(b-1)}{abqx-c}
\end{pmatrix},
$$
and
$$
C_3 := \begin{pmatrix}
\dfrac{x(c(a+b)-ab)-c^2}{cx}&\dfrac{c-abx}{x}\\
\dfrac{x-1}{x}&\dfrac{c-abx}{cx}
\end{pmatrix},
$$
which can be calculated using the {\tt RationalSolution} tool in the {\tt QDifferenceEquations} package of {\sc Maple}
(see also \cite[Ex.~2.35]{OW} for the case $a=b$ and $c=q$), satisfy~\eqref{eq:intdiffdiff1}. However,~\eqref{eq:intdiffdiff3} is satisfied only for the pair $C_1$ and $C_2$. 
So, the $q$-difference equation~\eqref{eq:HG} is difference isomonodromic over any $\Phi\cup\{\sigma_a,\sigma_b\}$-field containing $\K$. 

Applied to this situation, the proof of Theorem~\ref{thm:main} contains an algorithm that turns~\eqref{eq:HG} into a difference isomonodromic equation with respect to $\Phi\cup\Sigma$ just by extending $\K$ by at most one $\sigma_c$-transcendental element $x_1$ to $\K$, which is constant with respect to $\phi_q$. In particular, a calculation in {\sc Maple} using the above package shows that the solution space of~\eqref{eq:Z1} has dimension $d =1$ and 
\begin{equation}\label{eq:x1}
\sigma_a(x_1) =\dfrac{a-c}{qa-c}x_1 ,\quad \sigma_b(x_1)=\dfrac{b-c}{qb-c}x_1.
\end{equation}
Since $$x_1 := \dfrac{1}{(a-c)(b-c)}$$ satisfies~\eqref{eq:x1}, following the proof of Theorem~\ref{thm:main}, we replace $C_3$ by
$$
D_3 := \begin{pmatrix}
\dfrac{x(c(a+b)-ab)-c^2}{(a-c)(b-c)cx}&\dfrac{c-abx}{(a-c)(b-c)x}\\
\dfrac{x-1}{(a-c)(b-c)x}&\dfrac{c-abx}{(a-c)(b-c)cx}
\end{pmatrix}
$$
Thus, Eq.~\eqref{eq:HG} is difference isomonodromic with respect to $\Sigma$ over $\K$ with the matrices $C_1$, $C_2$, and $D_3$, and the proof of Theorem~\ref{thm:main} has helped us discover the latter matrix.
\end{example}

Let $L_1$ be a $\{\Phi\cup\Sigma,\Delta\}$-field and $L_2$ be a $\Sigma$-field containing $L_1^{\Phi,\Delta}$, then $R := L_1\otimes_{L_1^{\Phi,\Delta}} L_2$ is an integral domain \cite[Lem.~6.11]{CharlotteMichael} and, therefore, $L_1L_2 := \Quot(R)$ is a $\{\Phi\cup\Sigma,\Delta\}$-field with $\{\Phi,\Delta\}$-constants equal  $L_2$.

\section{Main result}\label{sec:mainresult}
This section contains the main result, Theorem~\ref{thm:main}, which allows one to reduce the number of integrability conditions to be tested. 

\begin{theorem}\label{thm:main} Let $A_1,\ldots,A_q \in \GL_n(\K)$ and $B_1,\ldots,B_m \in\Mn_n(\K)$ satisfy~\eqref{eq:intdifference},~\eqref{eq:intdifferential1}, and~\eqref{eq:intdifferential2}.  If there exist $C_1,\ldots,C_r\in \GL_n(\K)$ satisfying~\eqref{eq:intdiffdiff1} and~\eqref{eq:intdiffdiff2},
then there exist
\begin{enumerate}
\item[\rm 1)] a computable $\Sigma$-field $F$ generated over $\K^{\Phi,\Delta}$ by at most $(r-1)n^2$ elements and
\item[\rm 2)] $D_1,\ldots, D_r\in \GL_n(\K F)$ such that all integrability conditions are satisfied:
\begin{align}
\phi_j(D_i)A_j&=\sigma_i(A_j)D_i,&1\Le i\Le r,\ 1\Le j\Le q,\label{eq:intdiffdiff21}\\
\sigma_j(B_i)D_j&=\partial_i(D_j)+D_jB_i, &1\Le i\Le m,\ 1\Le j\Le r,\label{eq:intdiffdiff22}\\
\sigma_j(D_i)D_j&=\sigma_i(D_j)D_i, &1\Le i,j\Le r.\label{eq:intdiffdiff23}
\end{align} 
\end{enumerate}
\end{theorem}

\begin{proof}
This will be done by induction. Let there exist $C_1,\ldots, C_r \in\GL_n(\K)$ satisfying~\eqref{eq:intdiffdiff1} and~\eqref{eq:intdiffdiff2} and let $k$ be such that $2 \Le k \Le r$. Suppose we have already computed a $\Sigma$-field $F_{k-1}$ generated over $\K^{\Phi,\Delta}$ by at most  $(k-2)n^2$ elements and
$$
D_1,\ldots, D_{k-1} \in \GL_n\big(\K F_{k-1}\big)
$$ 
that satisfy~\eqref{eq:intdiffdiff21},~\eqref{eq:intdiffdiff22}, and~\eqref{eq:intdiffdiff23} ($k-1$ is substituted for $r$).  
 We claim that there exists a $\Sigma$-field $F_k$  generated  over $\K^{\Phi,\Delta}$ by at most $(k-1)n^2$ elements and $Z\in \GL_n(\K F_k)$ such that
$$
(D_1,\ldots,D_{k-1},D_k) := (D_1,\ldots, D_{k-1},ZC_k)
$$ 
satisfies~\eqref{eq:intdiffdiff21},~\eqref{eq:intdiffdiff22}, and~\eqref{eq:intdiffdiff23} ($k$ is substituted for $r$).
We need to construct a $\Sigma$-field $F_k$ and $Z \in\GL_n(\K F_k)$ such that 
\begin{align}\label{eq:55}
\sigma_k(A_i)ZC_k&=\phi_i(ZC_k)A_i,&&1\Le i\Le q,\\
\label{eq:6}
\sigma_k(B_i)ZC_k&=\partial_i(ZC_k)+ZC_kB_i, &&1\Le i\Le m,\\
\sigma_k(D_i)ZC_k&=\sigma_i(ZC_k)D_i, &&1\Le i\Le k-1.\label{eq:65}
\end{align}
Expanding the right-hand sides of~\eqref{eq:55},~\eqref{eq:6},~and~\eqref{eq:65} using~\eqref{eq:intdiffdiff1} and~\eqref{eq:intdiffdiff2},  we have 
\begin{align*}
\phi_i(ZC_k)A_i&=\phi_i(Z)\phi_i(C_k)A_i=\phi_i(Z)\sigma_k(A_i)C_k, & 1\Le i\Le q,\\
\partial_i(ZC_k)+ZC_kB_i&=\partial_i(Z)C_k+Z(\partial_i(C_k)+C_kB_i)=\partial_i(Z)C_k+Z\sigma_k(B_i)C_k,& 1\Le i \Le m,\\
\sigma_i(ZC_k)D_i&=\sigma_i(Z)\sigma_i(C_k)D_i,&1\Le i\Le k-1.
\end{align*} 
Therefore,~\eqref{eq:55},~\eqref{eq:6},~and~\eqref{eq:65} turn into
\begin{align}
\phi_i(Z)&=\sigma_k(A_i)Z\sigma_k(A_i)^{-1}, & 1\Le i\Le q,\label{eq:Z1}\\
\partial_i(Z)&=[\sigma_k(B_i),Z],& 1\Le i \Le m,\label{eq:Z2}\\
\sigma_i(Z)&=\left(\sigma_k(D_i)Z\sigma_k(D_i)^{-1}\right)\left(\sigma_k(D_i)C_kD_i^{-1}\sigma_i(C_k)^{-1}\right),&1\Le i\Le k-1.\label{eq:Z3}
\end{align} 
Let us now demonstrate how find an invertible solution of~\eqref{eq:Z1},~\eqref{eq:Z2}, and~\eqref{eq:Z3}.
For this, let $V$ be the $F_{k-1}$-vector space of matrix solutions of~\eqref{eq:Z1} and~\eqref{eq:Z2} in $\Mn_n(\K F_{k-1})$ and let $$\dim_{F_{k-1}}V=:d \Le n^2.$$ We will show that $V$ is invariant under the vector space isomorphisms defined by~\eqref{eq:Z3}, and then will construct a solution of the restriction of~\eqref{eq:Z3} to $V$. The former will consist of several steps. At the {\bf first step}, choose an $F_{k-1}$-basis 
\begin{equation}\label{eq:basis}
\left\{
\begin{pmatrix}
z_{11,1}\\
\vdots\\
z_{n1,1}\\
\vdots\\
z_{1n,1}\\
\vdots\\
z_{nn,1}
\end{pmatrix},\ldots,
\begin{pmatrix}
z_{11,d}\\
\vdots\\
z_{n1,d}\\
\vdots\\
z_{1n,d}\\
\vdots\\
z_{nn,d}
\end{pmatrix}
\right\}
\end{equation}
 of $V$ so that
 \begin{equation}\label{eq:Zspecial}
 \begin{pmatrix}
 z_{11,1}&\ldots&z_{1n,1}\\
 \vdots&\ddots&\vdots\\
 z_{n1,1}&\ldots&z_{nn,1}
 \end{pmatrix} = \sigma_1^{-1}\left(\sigma_k(D_1)C_kD_1^{-1}\sigma_1(C_k)^{-1}\right) \in \GL_n(\K F_{k-1}).
 \end{equation}
This matrix is indeed in $V$, which also implies that $d \Ge 1$. To prove the former (and also for the purposes of our construction that will follow), let us show that, for all $i$, $1\Le i \Le k-1$, 
\begin{equation}\label{eq:sateq}
H_i := \sigma_i^{-1}\left(\sigma_k(D_i)C_kD_i^{-1}\sigma_i(C_k)^{-1}\right)
\end{equation}
 satisfies~\eqref{eq:Z1} as an equation in $Z$. Indeed, for all $j$, $1\Le j \Le q$, we have
\begin{align*}
\phi_j\left(\sigma_k(D_i)C_kD_i^{-1}\sigma_i(C_k)^{-1}\right)&=\sigma_k\big(\phi_j(D_i)\big)\phi_j(C_k)\phi_j(D_i)^{-1}\sigma_i\big(\phi_j(C_k)\big)^{-1}\\
&=\sigma_k\big(\phi_j(D_i)A_j\big)\sigma_k(A_j)^{-1}\phi_j(C_k)\phi_j(D_i)^{-1}\sigma_i\big(\phi_j(C_k)\big)^{-1} \\
&=\sigma_i\big(\sigma_k(A_j)\big)\sigma_k(D_i)\sigma_k(A_j)^{-1}\phi_j(C_k)\phi_j(D_i)^{-1}\sigma_i\big(\phi_j(C_k)\big)^{-1} \\
&=\sigma_i\big(\sigma_k(A_j)\big)\sigma_k(D_i)C_kA_j^{-1}\phi_j(D_i)^{-1}\sigma_i\big(\phi_j(C_k)\big)^{-1}\\
&=\sigma_i\big(\sigma_k(A_j)\big)\sigma_k(D_i)C_kD_i^{-1}\sigma_i(A_j)^{-1}\sigma_i\big(\phi_j(C_k)\big)^{-1}\\
&=\sigma_i\big(\sigma_k(A_j)\big)\sigma_k(D_i)C_kD_i^{-1}\sigma_i(C_k)^{-1}\sigma_i\big(\sigma_k(A_j)\big)^{-1},
\end{align*}
which implies the claim. Moreover, we will prove that~\eqref{eq:sateq} satisfies~\eqref{eq:Z2} as an equation in $Z$. Indeed, for all $j$, $1\Le j \Le m$,
\begin{align*}
\sigma_i&\big(\sigma_k(B_j)\big)\sigma_k(D_i)C_kD_i^{-1}\sigma_i(C_k)^{-1}-\sigma_k(D_i)C_kD_i^{-1}\sigma_i(C_k)^{-1}\sigma_i\big(\sigma_k(B_j)\big)\\
&=\sigma_k\big(\partial_j(D_i)+D_iB_j\big)C_kD_i^{-1}\sigma_i(C_k)^{-1}-\sigma_k(D_i)C_kD_i^{-1}\sigma_i(C_k)^{-1}\sigma_i\big(\sigma_k(B_j)\big)\\
&=\partial_j(\sigma_k(D_i))C_kD_i^{-1}\sigma_i(C_k)^{-1}+\sigma_k(D_i)\sigma_k(B_j)C_kD_i^{-1}\sigma_i(C_k)^{-1}\\
&\quad-\sigma_k(D_i)C_kD_i^{-1}\sigma_i\left(C_k^{-1}\sigma_k(B_j)\right)\\
&=\partial_j(\sigma_k(D_i))C_kD_i^{-1}\sigma_i(C_k)^{-1}+\sigma_k(D_i)\sigma_k(B_j)C_kD_i^{-1}\sigma_i(C_k)^{-1}\\
&\quad-\sigma_k(D_i)C_kD_i^{-1}\sigma_i\left(C_k^{-1}\partial_j(C_k)C_k^{-1}+B_jC_k^{-1}\right)\\
&=\partial_j(\sigma_k(D_i))C_kD_i^{-1}\sigma_i(C_k)^{-1}+\sigma_k(D_i)\sigma_k(B_j)C_kD_i^{-1}\sigma_i(C_k)^{-1}\\
&\quad-\sigma_k(D_i)C_kD_i^{-1}\sigma_i(B_j)\sigma_i(C_k)^{-1}+\sigma_k(D_i)C_kD_i^{-1}\partial_j\left(\sigma_i(C_k)^{-1}\right)\\
&=\partial_j(\sigma_k(D_i))C_kD_i^{-1}\sigma_i(C_k)^{-1}+\sigma_k(D_i)\big(\partial_j(C_k)+C_kB_j\big)D_i^{-1}\sigma_i(C_k)^{-1}\\
&\quad-\sigma_k(D_i)C_kD_i^{-1}\sigma_i(B_j)\sigma_i(C_k)^{-1}+\sigma_k(D_i)C_kD_i^{-1}\partial_j\left(\sigma_i(C_k)^{-1}\right)\\
&=\partial_j(\sigma_k(D_i))C_kD_i^{-1}\sigma_i(C_k)^{-1}+\sigma_k(D_i)\partial_j(C_k)D_i^{-1}\sigma_i(C_k)^{-1}\\
&\quad+\sigma_k(D_i)C_k\left(B_jD_i^{-1}-D_i^{-1}\sigma_i(B_j)\right)\sigma_i(C_k)^{-1}+\sigma_k(D_i)C_kD_i^{-1}\partial_j\left(\sigma_i(C_k)^{-1}\right)\\
&=\partial_j(\sigma_k(D_i))C_kD_i^{-1}\sigma_i(C_k)^{-1}+\sigma_k(D_i)\partial_j(C_k)D_i^{-1}\sigma_i(C_k)^{-1}\\
&\quad+\sigma_k(D_i)C_k\partial_j\left(D_i^{-1}\right)\sigma_i(C_k)^{-1}+\sigma_k(D_i)C_kD_i^{-1}\partial_j\left(\sigma_i(C_k)^{-1}\right)\\
&=\partial_j\left(\sigma_k(D_i)C_kD_i^{-1}\sigma_i(C_k)^{-1}\right).
\end{align*}
Our {\bf second step} is to show that, for any $\{\Phi\cup\Sigma,\Delta\}$-ring $R$, $V$ is invariant under the invertible map $$L_i : \Mn_n(R)\to\Mn_n(R),\quad Z \mapsto \sigma_i^{-1}\left(\sigma_k(D_i)Z\sigma_k(D_i)^{-1}\right),\quad 1\Le i < k.$$
Indeed, for all $j$, $1\Le j\Le q$
\begin{align*}
\sigma_i\left(\phi_j\left(\sigma_i^{-1}\left(\sigma_k(D_i)Z\sigma_k(D_i)^{-1}\right)\right)\right)&=\sigma_k(\phi_j(D_i))\sigma_k(A_j)Z\sigma_k\left(A_j^{-1}\phi_j(D_i)^{-1}\right)\\
&=\sigma_k\big(\sigma_i(A_j)D_i\big)Z\sigma_k\left(D_i^{-1}\sigma_i(A_j)^{-1}\right)\\
&=\sigma_i\left(\sigma_k(A_j)\sigma_i^{-1}\left(\sigma_k(D_i)Z\sigma_k(D_i)^{-1}\right)\sigma_k(A_j)^{-1}\right).
\end{align*}
Moreover, for all $j$, $1\Le j\Le m$,
\begin{align*}
&\left[\sigma_i(\sigma_k(B_j)),\sigma_k(D_i)Z\sigma_k(D_i)^{-1}\right]=\sigma_i\big(\sigma_k(B_j)\big)\sigma_k(D_i)Z\sigma_k(D_i)^{-1}-\sigma_k(D_i)Z\sigma_k(D_i)^{-1}\sigma_i(\sigma_k(B_j))\\
&=\sigma_k\big(\partial_j(D_i)+D_iB_j\big)Z\sigma_k(D_i)^{-1}-\sigma_k(D_i)Z\sigma_k\left(D_i^{-1}\partial_j(D_i)D_i^{-1}+B_jD_i^{-1}\right)\\
&=\partial_j(\sigma_k(D_i))Z\sigma_k(D_i)^{-1}+\sigma_k(D_i)\big[\sigma_k(B_j),Z\big]\sigma_k(D_i)^{-1}-\sigma_k(D_i)Z\sigma_k\left(D_i^{-1}\partial_j(D_i)D_i^{-1}\right)\\
&=\partial_j\big(\sigma_k(D_i)\big)Z\sigma_k(D_i)^{-1}+\sigma_k(D_i)\partial_j(Z)\sigma_k(D_i)^{-1}- \sigma_k(D_i)Z\sigma_k\left(D_i^{-1}\partial_j(D_i)D_i^{-1}\right)\\
&=\partial_j\left(\sigma_k(D_i)Z\sigma_k(D_i)^{-1}\right).
\end{align*}
We have shown in the above that, for all $i$, $1\Le i <k$, $V$ is $L_i$-invariant and $H_i \in V$ (see~\eqref{eq:sateq}). 
In our {\bf third step}, we will show that, for all $i$, $1\Le i < k$, $V$ is invariant under the invertible map $$M_i : \Mn_n(R) \to \Mn_n(R),\quad Z \mapsto L_i(Z)H_i.$$
Indeed, for all $Z \in V$ and $j$, $1\Le j\Le q$, 
\begin{align*}
\phi_j\left(L_i(Z)H_i\right)&=\phi_j(L_i(Z))\phi_j(H_i)=\sigma_k(A_i)L_i(Z)\sigma_k(A_i)^{-1}\sigma_k(A_i)H_i\sigma_k(A_i)^{-1}\\
&=\sigma_k(A_i)L_i(Z)H_i\sigma_k(A_i)^{-1},
\end{align*}
and, for all $j$, $1\Le j\Le m$,
\begin{align*}
\partial_j\left(L_i(Z)H_i\right)&=\partial_j\big(L_i(Z)\big)H_i+L_i(Z)\partial_j(H_i)=\big[\sigma_k(B_j),L_i(Z)\big]H_i+L_i(Z)\big[\sigma_k(B_j),H_i\big]\\
&=\left[\sigma_k(B_j),L_i(Z)H_i\right].
\end{align*}
{\bf Finally}, since, for every $i$, $1\Le i \Le k-1$, the $F_{k-1}$-linear map $\sigma_i\circ M_i$ is invertible and sends $V$ into $\sigma_i(V)$ by the above, in the basis~\eqref{eq:basis}, Eq.~\eqref{eq:Z3} is of the form
$$
\sigma_i(Y) = E_iY,\quad E_i\in\GL_d\big(F_{k-1}\big),\ 1\Le i\Le k-1.
$$ 
We now let 
\begin{equation}\label{eq:defsigmaxi}
\Sigma_k = \{\sigma_k,\ldots,\sigma_r\},\ \  F_k := F_{k-1}\big\{x_s,\ 1\Le s\Le d\big\}_{\Sigma_k},\ \  \sigma_i\left(\begin{pmatrix}
x_1\\
\vdots\\
x_d
\end{pmatrix}\right):=E_i\cdot\begin{pmatrix}
x_1\\
\vdots\\
x_d
\end{pmatrix},\ \  1\Le i\Le k-1,
\end{equation}
where $x_s$, $1\Le s\Le d$, are $\Sigma_k$-indeterminates. We need to prove that the action of $\sigma_i$, $1\Le i\Le k-1$, is well-defined. For this, it is sufficient to show that, for all $u$ and $v$, $1\Le u,v\Le k-1$,
$$
\sigma_u(E_v)E_u=\sigma_v(E_u)E_v.
$$ 
For this, it is sufficient to show that~\eqref{eq:Z3} satisfies the integrability conditions. For all $u$ and $v$, $1\Le u,v\Le k-1$, we have
\begin{align*}
&\sigma_u\big(\sigma_k(D_v)Z\big)\sigma_u\left(C_kD_v^{-1}\sigma_v(C_k)^{-1}\right)\\
&=\sigma_k\left(\sigma_v(D_u)D_vD_u^{-1}\right)\sigma_k(D_u)ZC_kD_u^{-1}\sigma_u(C_k)^{-1}\sigma_u\left(C_kD_v^{-1}\sigma_v(C_k)^{-1}\right)\\
&=\sigma_v\big(\sigma_k(D_u)\big)\sigma_k(D_v)ZC_kD_u^{-1}\sigma_u(D_v)^{-1}\sigma_v(\sigma_u(C_k))^{-1}\\
&=\sigma_v\big(\sigma_k(D_u)Z\big)\sigma_v(C_k)D_vD_u^{-1}\sigma_u(D_v)^{-1}\sigma_v\big(\sigma_u(C_k)\big)^{-1}\\
&=\sigma_v\left(\sigma_k(D_u)Z\right)\sigma_v\left(C_kD_u^{-1}\sigma_u(C_k)^{-1}\right).
\end{align*}
It remains to note that, since $x_1,\ldots,x_d$ are transcendental over $\K F_{k-1}$, the matrix 
$$
Z := \begin{pmatrix}
z_{11,1}x_1+\ldots + z_{11,d}x_d&\ldots&z_{1n,1}x_1+\ldots + z_{1n,d}x_d\\
\vdots&\ddots&\vdots\\
z_{n1,1}x_1+\ldots + z_{n1,d}x_d&\ldots&z_{nn,1}x_1+\ldots + z_{nn,d}x_d
\end{pmatrix}
\in \GL_n(\K F_k)
$$
 by~\eqref{eq:Zspecial}. Indeed, $\det(Z)$ is a polynomial in the indeterminates $x_1,\ldots,x_d$ and takes the value $\det(H_1)\ne 0$ under the substitution $x_1=1,x_2=0,\ldots,x_d=0$ (see~\eqref{eq:Zspecial} and~\eqref{eq:sateq}). Therefore, $\det(Z) \ne 0$.
 Finally, the matrix $Z$ also satisfies~\eqref{eq:Z1},~\eqref{eq:Z2}, and~\eqref{eq:Z3} by construction (see~\eqref{eq:basis} and~\eqref{eq:defsigmaxi}), which finishes the proof.
\end{proof}

\section{Acknowledgments}
The author is grateful to S. Abramov, S. Chen, and the referee for their helpful suggestions.
This work has been partially supported by the NSF grant CCF-0952591.

\bibliographystyle{model1b-num-names}
\bibliography{diffint-final}
\end{document}